\documentclass[a4paper,12pt]{amsart}
\usepackage[colorlinks=true,urlcolor=blue,citecolor=blue,linkcolor=blue,pdfstartview=FitH,pdfpagemode=None,bookmarksnumbered=true]{hyperref}
\usepackage{graphicx}

\hypersetup{
    pdftitle={Periodic orbits of period 3 in the disc}, 
    pdfauthor={Boris Kolev}, 
    pdfkeywords={Dynamics of surfaces homeomorphisms; periodic orbits}, 
    pdfsubject={1991 Mathematics Subject Classification: 54H20, 58F13, 34C35}, 
    }
    
\newtheorem{thm}{Theorem}[section]

\newtheorem{lem}[thm]{Lemma}

\theoremstyle{definition}

\newtheorem*{rmk}{Remark}

\theoremstyle{remark}

\numberwithin{equation}{section}

\begin{document}

\title{Periodic orbits of period $3$ in the disc}
\author{Boris Kolev}%
\date{28 January 1994}

\address{CMI, 39 rue F. Joliot-Curie, 13453 Marseille cedex 13, France}
\email{kolev@cmi.univ-mrs.fr}

\subjclass{54H20, 58F13, 34C35}%
\keywords{Dynamics of surfaces homeomorphisms; periodic orbits}%

\maketitle


\begin{abstract}
Let $f$ be an orientation preserving homeomorphism of the disc
$D^{2}$ which possesses a periodic point of period 3. Then either
$f$ is isotopic, relative the periodic orbit, to a homeomorphism $g$
which is conjugate to a rotation by $2\pi/3$ or $4\pi/3$, or $J$ has
a periodic point of least period $n$ for each $n\in \mathbb{N}^{*}$.
\end{abstract}


\section{Introduction}

Some very famous results of one dimensional dynamics state that a
continuous map of the interval which has a periodic point of period
3 has periodic points of all other periods \cite{LY75,Sar64} and
positive topological entropy \cite{BF76}. These results are
certainly not true in higher dimensions; a rotation by $2\pi/3$ in
the disc $D^{2}$ has periodic points of period 3 and 1 only and
topological entropy zero. In fact the period alone is not enough to
specify a periodic orbit of a surface's homeomorphism $f$ in order
to get information about the structure of the set of periodic points
of $f$.

What sort of stronger hypothesis do we have to add in order to get
similar results in dimension two ? This problem has been extensively
studied these last decades and a lot of results have emerged.
Blanchard and Franks \cite{BF80} have shown that an orientation
reversing homeomorphism $f$ of the $2$-sphere which has a point of
period $p$ and a point of period $q$ where $p$ and $q$ are distinct
odd integers has positive topological entropy. If in addition, $f$
is $C^{1+\epsilon}$ \cite{Kat80}, $f$ has periodic points of
arbitrary large period. Handel \cite{Han82} has generalized this
result to orientation reversing homeomorphisms of any compact
surface $M$ using Nielsen-Thurston theory. Following an idea of
Bowen \cite{Bow78}, these authors specify a periodic point $z$ by
the isotopy class of the restriction $f_{z}$ of $f$ to $M_{z}=M-o(z,
f)$; the complement in $M$ of the orbit of $z$ under $f$. This idea
has since been formalized by Boyland \cite{Boy94} who introduced the
`braid type' of a periodic orbit. Specifically, if $[f_{z}]$ denotes
the isotopy class of $f_{z}$ in $M_{z}$, define the braid type of
$o(z, f)$, denoted $bt(z, f)$, to be the conjugacy class of
$[f_{z}]$ in the mapping class group of $M_{z}$. Recently several
authors \cite{GST89,LM90} have characterized all the braid types
which appear for diffeomorphisms of genus zero surfaces with
topological entropy zero.

More recently, Gambaudo, van Strien and Tresser \cite{GST90} have
given an example of a braid type of period 3 in the disc which imply
all other periods and positive topological entropy, which is the
exact analogue of the result in one dimensional dynamics stated
previously. In this paper we show that this is a quite general
result, namely:

\begin{thm}\label{Main_theorem}
If $f$ is an orientation preserving homeomorphism of the disc
$D^{2}$ which has a periodic point $z$ of period 3, then either $f$
is isotopic to a homeomorphism $g$ which is conjugate to a rotation
by $2\pi/3$ or $4\pi/3$ or $f$ has periodic points of least period
$n$ for each $n\in \mathbb{N}^{*}$.
\end{thm}

In short terms, this result means that pseudo-Anosov $3$ braid types
of the disc imply all other periods. By the way, this result is
obviously not true for the only two other $3$-braid types which are
not pseudo-Anosov as cited previously. Therefore, it should be
emphasized that Theorem~\ref{Main_theorem} is the best possible
extension of the paper. ``Period three implies chaos" \cite{LY75}
for the class of interval maps to the class of orientation
preserving homeomorphisms of the disc $D^{2}$.

The proof of this theorem is based on the structure of pseudo-Anosov
homeomorphisms in the disc with three holes and upon some results of
John Franks about the existence of periodic points for
homeomorphisms of the annulus \cite{Fra88}. In section 2, we recall
the basic facts we need and give the proof of
Theorem~\ref{Main_theorem} in section 3.

\section{Background and definitions}

\subsection{Pseudo-Anosov homeomorphisms}

We begin by recalling some properties of pseudo-Anosov
homeomorphisms that we need and that can be found in \cite{FL83} or
\cite{FLP79} (see also \cite{Jia83} for the definition of the fixed
point index). A pseudo-Anosov homeomorphism $\phi$ of a compact
orientable manifold $M$ with negative Euler characteristic is a
homeomorphism such that there exists a pair of transverse measured
foliations $F^{u}$ and $\mathcal{F}^{s}$ and a number $\lambda>1$
with $\displaystyle \phi(\mathcal{F}^{u})=
\lambda^{-1}\mathcal{F}^{u}$ and $\phi(\mathcal{F}^{s})=\lambda
\mathcal{F}^{s}$.

\subsubsection{}\label{211}
Each of the foliation have a finite number of singularities and at
least one on each component of the boundary of $M$. These
singularities are one-prong singularities on the boundary and
$p$-prongs singularities with $p\geq 3$ in the interior of {\it M}.
$\mathcal{F}^{u}$ and $\mathcal{F}^{s}$ have the same singularities
in the interior of $M$ and on each component of the boundary the
singularities of $\mathcal{F}^{u}$ and $\mathcal{F}^{s}$ alternate.
Moreover, we have the following formula for each foliation
\begin{equation*}
\sum_{s}(2-p_{s})=2\chi(M)
\end{equation*}
where the sum is taken over all the singularities $s,$ $p_{s}$ is
the number of prongs at $s$ if $s$ lies in the interior of $M$ and
$p_{s}=3$ if $s$ is a one-prong singularity on the boundary.

\subsubsection{}\label{212}
Let $x$ be a fixed point of $\phi^{n}$. Then $Ind(x,\phi^{n})\neq
0,$ $Ind(x, \phi^{n})\leq 1$ and $Ind(x, \phi^{n})<0$ if $x$ lies on
the boundary of $M$.

\subsubsection{}\label{213}
Let $x$ be a fixed point of $\phi$.
\begin{itemize}

 \item If $Ind(x, \phi)=-1$, then $Ind(x, \phi^{n})=-1$ for all $n$.

 \item If $Ind(x, \phi)=1$, then there exist two positive integers $p$ and $q$ such that:
 \begin{equation*}
    Ind(x, \phi)=Ind(x.\phi^{2})=\ldots=Ind(x, \phi^{q-1})=1
 \end{equation*}
 and
 \begin{equation*}
 Ind(x, \phi^{q})=1-p.
 \end{equation*}
 In that case, there are exactly $p$ stable (resp. unstable)
 half-leaves starting at $x$ and the permutation induced by $\phi$ on
 theses half-leaves consists of $p/q$ $q$-cycles (if $p\geq 3,$ $x$
 is a singularity of the foliations).

\end{itemize}

\subsubsection{}\label{214}
If $\phi$ has a periodic point of least period $n$ which is not on
the boundary of $M$, then every map $g$ isotopic to $\phi$ must
possess a periodic point of least period $n$.

\subsubsection{}\label{215}
$\phi$ has a dense orbit.

\subsection{Homeomorphisms of the annulus}

Next, we briefly review some results about homeomorphisms of the
annulus and chain transitivity [9]. Let $f$ $:$ $A$ $\rightarrow A$
a homeomorphism of the annulus $A=S^{1}\times[0,1]$ which is
orientation preserving and boundary component preserving.Let $\pi$
$:$ $\tilde{A}$ $\rightarrow A$ be the universal cover,
$\tilde{A}=\mathbb{R}\times[0,1]$ and $p_{1}$ :
$\tilde{A}\rightarrow \mathbb{R}$ the projection onto the first
coordinate.$ $For each $x\in A$, choose a lift
$\tilde{x}\in\tilde{A}$ and consider:
\begin{equation*}
\lim_{n\to\infty}\frac{p_{1}(\tilde{f}^{n}(\tilde{x}))-p_{1}(\tilde{x})}{n}
\end{equation*}
where $\tilde{f}$ is a lift of $f$. When this limit exists, it is
called the $\tilde{f}$-rotation number of $x$ and denoted
$\rho(\tilde{f}, x)$. It is independent of the choice of the lift
$\tilde{x}$of $x$. If $x$ belongs to one component of the boundary,
$\rho(x,\tilde{f})$ is defined and independent of the choice of $x$
on this component. Let $\rho_{1}(\tilde{f})$ and
$\rho_{2}(\tilde{f})$ these two numbers. A periodic point $z$ is
called a $p/q$ periodic point relative to $\tilde{f}$ if $z$ has
least period $q$ and $\tilde{f}^{q}(\tilde{z})=\tilde{z}+(p, 0)$ for
some lift (and hence all) $\tilde{z}$ of $z$.

A homeomorphism $f$ is called chain-transitive if for any $x,$ $y\in
A$ and $\epsilon>0$ there exists a sequence $x_{0}=x,$ $x_{1},$
$\ldots$ $,$ $x_{n}=y,$ $n\geq 1$ such that $ d(f(x_{i}),
x_{i+1})<\epsilon$. In particular if $f$ has a dense orbit it is
chain transitive.

\begin{thm}[Franks]\label{Franks_theorem}
Let $f$ : $A\rightarrow A$ be an orientation preserving, boundary
component preserving homeomorphism which is chain transitive and
$\tilde{f}$ : $\tilde{A}\rightarrow\tilde{A}$ a lift of $f$. Then
for every rational $p/q$ between $\rho_{1}(\tilde{f})$ and
$\rho_{2}(\tilde{f}),$ $f$ has a $p/q$ periodic point.
\end{thm}

\section{Proof of main theorem}

We write $z_{i}=f^{i}(z),$ $(i=1,2,3)$ for the points of the orbits
of $z$ and set $o(z)=\{z_{1}, z_{2}, z_{3}\}$. Let
$D_{z}=D^{2}-o(z),$ $M$ be the compact surface obtained from $D_{z}$
by compactifying each puncture $z_{i}$ by a circle $S_{i}$ and $\pi$
$:$ $M$ $\rightarrow D^{2}$ be the map which collapses each circle
$S_{i}$ to $z_{i}$.

There is a natural isomorphism between the mapping class group of
$D_{z}$ and the mapping class group of $M$. In our case $[f_{z}]$ is
not reducible (in the sense of Nielsen-Thurston) because $z$ is a
point of period 3 and there are no reducible braid type of prime
period \cite{Boy94}. Consequently $[f_{z}]$ is either of finite
order or of pseudo-Anosov type. If $[f_{z}]$ is of finite order,
then there exists a homeomorphism $g$ in $D^{2}$ which is isotopic
to $f$ rel. $o(z)$ and satisfy $g^{3}=Id$. Consequently, by a
theorem of Ker\'{e}kj\'{a}rt\'{o} \cite{CK94,Eil34}, $g$ is conjugate to a
rotation by $2\pi/3$ or $4\pi/3$ and we are done. If $[f_{z}]$ is of
pseudo-Anosov type, we are going to show that $f$ possesses a
periodic point of least period $n$ for each integer $n$. Let $\phi$
be a pseudo-Anosov homeomorphism of $M$ in the class $[f_{z}]$.

\begin{lem}\label{Main_lemma}
$\phi$ has a periodic point of least period $n$ in the interior of
$M$ for each $n\in \mathbb{N}^{*}\backslash\{3\}$.
\end{lem}

\begin{proof}[Proof of Lemma~\ref{Main_lemma}]
$\chi(M)=-2$ ($M$ is a sphere with $4$ holes !), so for each
foliation $\mathcal{F}^{u}$ and $\mathcal{F}^{s}$ there is exactly
one singularity on each boundary component of $M$ and no singularity
in the interior of $M$ (\ref{211}). Let $S_{\infty}$ be the
component of $\partial M$ corresponding to the boundary of $D^{2}$.
$S_{\infty}$ is invariant under $\phi$ and the singularities of the
two foliations on $S_{\infty}$ are fixed points for $\phi$ (see
\autoref{fig}).

Let $L(\phi)$ be the Lefschetz number of $\phi$ (see \cite{Jia83}
for definition and computation of the Lefschetz number). We have:
\begin{equation*}
L(\phi)=\sum_{x\in Fix(\phi)}Ind(x, \phi)=1-Tr(\phi_{1}^{*})=1,
\end{equation*}
where $\phi_{1}^{*}$ is the action induced by $\phi$ on the first
homology group of $M$. Therefore, by (\ref{212}), $\phi$ has a fixed
point $z_{0}$ such that $Ind(z_{0}, \phi)=1$. It results from
(\ref{212}) that $z_{0}$ lies in the interior of $M$ and thus is not
a singularity of the foliations. For the same reason
$(L(\phi^{2})=1),$ $\phi^{2}$ has a a fixed point of index $1$ in
the interior of $M$ which is not a fixed point of $\phi$
(\ref{213}). Therefore $\phi$ has a periodic point of period 2 in
the interior of $M$.

Blow up $z_{0}$ into a circle $S_{0}$ (there is no problem to extend
the pseudo-Anosov $\phi$ to $S_{0}$) and collapse the three circles
$S_{1},$ $S_{2},$ $S_{3}$ to the original points $z_{1},$ $z_{2},$
$z_{3}$. Using the fact that $\phi$ has a dense orbit (\ref{215}),
we obtain by this way a chain-transitive homeomorphism of a closed
annulus $A$ that we shall call $\Phi$.

Because $z_{0}$ is not a singularity of the foliation, we have
$Ind(z_{0}, \phi^{2})=-1$ by (\ref{213}) and thus $\Phi$ has two
period 2 periodic orbits on $S_{0}$.

Let $\tilde{\Phi}$ be any lift of $\Phi$ to the universal cover
$\tilde{A}$ of $A$ and call $\rho_{0}(\tilde{\Phi})$ (resp.
$\rho_{\infty}(\tilde{\Phi})$) the $\tilde{\Phi}$-rotation number of
$\Phi/s_{0}$ (resp. $\Phi/s_{\infty}$). Since $\Phi$ has a fixed
point on $S_{\infty}$ and a periodic point of. period 2 on $S_{0},$
$\rho_{\infty}(\tilde{\Phi})=k$ for some $k\in \mathbb{Z}$ and
$\rho_{0}(\tilde{\Phi})=(2l+1)/2$ for some $l\in \mathbb{Z}$. Using
Theorem~\ref{Franks_theorem} we conclude that$\Phi$ has a periodic
point of least period $n$ for each $n\geq 2$. All the periodic
points of period $n\geq 3$ belong to the interior of $A$ and
correspond therefore to periodic points of $\phi$.
\end{proof}

\begin{figure}
\begin{center}
\includegraphics{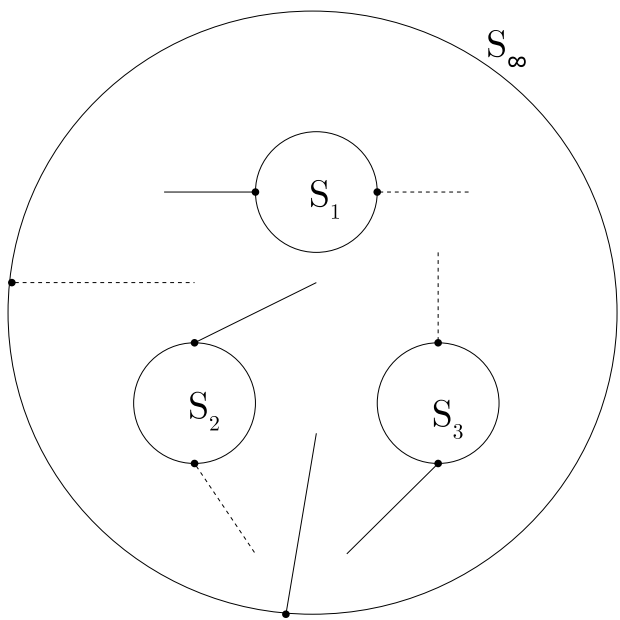}
\caption{The manifold $M$} \label{fig}
\end{center}
\end{figure}

Let us now give the proof of Theorem~\ref{Main_theorem}. Suppose
first that $f$ is continuously differentiable around the points of
$\mathrm{o}(z)$. Then we can extend $f_{z}$ to a homeomorphism $F$
$:$ $M\rightarrow M$ which is isotopic to $\phi$ and there is a one
to one correspondence between periodic points of $F$ in $int(M)$ and
periodic points of $f_{\mathrm{Z}}$ in $int(D_{z})$. Using the
previous lemma and (\ref{214}) $f$ has a periodic point of least
period $n$ for every $n\in \mathbb{N}^{*}$.

In the case where $f$ is not supposed to be well-behaved near
$\mathrm{o}(z)$, we construct a sequence $f_{k}$ of homeomorphisms
of $D^{2}$ which coincide with $f$ away from a
$\mathrm{l}/k$-neighbourhood of $\mathrm{o}(z)$, which act on
$\{z_{1}, z_{2}, z_{3}\}$ as $f$ but which are continuously
differentiable around these points. Each $f_{k}$ extends to a
homeomorphism $F_{k}$ of $M$ which is isotopic to $\phi$ in $M$ and
\cite{Han90}:

\begin{lem}[Handel]\label{Handel_lemma}
For every $\epsilon>0$, there exists $\delta>0$ (independent of $k$)
such that if $\phi$ has a periodic point $x$ of least period $n$ and
$ d(o(\phi, x), \partial M)>\epsilon$ then for every $k,$ $F_{k}$
has a periodic point $x_{k}$ of least period $n$ such that $
d(o(F_{k}, x_{k}),
\partial M)>\delta$.
\end{lem}

Therefore, $f$ has a periodic point of every period. $\square $

\begin{rmk}
If we consider no longer a periodic point of period 3 but a periodic
point of period $q>3$ such that $bt(z, f)$ is of type pseudo-Anosov
and if the singularities of the foliations of a pseudo-Anosov
representative $\phi$ on $S_{\infty}$ are fixed points of $\phi$,
then the proof of Theorem~\ref{Main_theorem} holds. There exists a
fixed point $z_{0}$ of index $1$ for $\phi$ and such that when we
blow up $z_{0}$ into a circle $S_{0}$ we get
$\rho(\phi_{/S_{0}})=k/l$ where $k$ and $l$ are coprime with $2\leq
f\leq q-2$. Therefore, we can prove that $f$ has periodic points of
period $n$ for every $n\geq q-2$. The general situation seems to be
much more complex.
\end{rmk}


\end{document}